\documentclass[11pt]{amsproc}

\usepackage{mathrsfs, color}
\usepackage{stmaryrd}
\usepackage{cases}
\usepackage{amssymb}
\usepackage{amsmath}
\usepackage{amsfonts}
\usepackage{graphicx}
\usepackage{amsmath,amstext,amsbsy,amssymb, color}
\usepackage[colorlinks=true, linkcolor=blue, citecolor=magenta, urlcolor=blue, backref=page]{hyperref}
\usepackage{latexsym}
\usepackage{amsmath}
\usepackage{amssymb}
\usepackage{bm}
\usepackage{graphicx}
\usepackage{wrapfig}
\usepackage{fancybox}

\newtheorem{theorem}{Theorem}[section]
\newtheorem{lemma}[theorem]{Lemma}
\newtheorem{conjecture}[theorem]{Conjecture}

\theoremstyle{definition}

\theoremstyle{remark}
\newtheorem{remark}[theorem]{Remark}

\newcommand{\diag}{\mbox{diag}}

\newcommand{\Imm}{\mathrm{Imm}}

\newcommand{\ot}{\otimes}

\newcommand{\si}{\sigma}
\newcommand{\GL}{\mathrm{GL}}

\newcommand{\Mat}{\mathrm{Mat}}

\newcommand{\Tc}{\mathcal{T}}
\newcommand{\bal}{\begin{aligned}}
\newcommand{\eal}{\end{aligned}}
\newcommand{\beq}{\begin{equation}}
\newcommand{\eeq}{\end{equation}}
\newcommand{\ben}{\begin{equation*}}
\newcommand{\een}{\end{equation*}}
\newcommand{\Sym}{\mathfrak S}
\newcommand{\CC}{\mathbb{C}}
\begin{document}
	
	\title[Immanant inequalities and weight spaces]
	{Immanant inequalities and weight spaces}
	
	\author{Naihuan Jing}
	\address{Department of Mathematics, North Carolina State University, Raleigh, NC 27695, USA}
	\email{jing@ncsu.edu}
	
		\author{Yinlong Liu}
	\address{Department of Mathematics, Shanghai University, Shanghai 200444, China}
	\email{yinlongliu@shu.edu.cn}
	
	\author{Jian Zhang}
	\address{School of Mathematics and Statistics,
		Central China Normal University, Wuhan, Hubei 430079, China}
	\email{jzhang@ccnu.edu.cn}
	
	\thanks{{\scriptsize
			\hskip -0.6 true cm MSC (2020): Primary: 20G05 Secondary: 15A15, 17B35, 05E10
	}}

\begin{abstract}
 We first obtain a trace formula for immanants of generalized principal submatrix of any complex matrix based on any weight space for finite dimensional representations of  the general linear group. Our trace formula contains Kostant's famous formula for immanants on $0$-weight spaces as special case. We then present a criterion for non-vanishing immanants for any generalized principal submatrix of positive definite Hermitian  or nonsingular totally nonnegative matrices, which strengthened the well-known results of Schur and Stembridge. Furthermore, we
 present an inequality that contains Kostant, Schur and Stembridge's famous inequalities as special cases.
\end{abstract}
	\maketitle
\section{Introduction}

In this paper, we will prove three results on the immanants of postive semidefinite matrices or totally nonnegative matrices
which generalize and strengthen some well-known results of Schur, Stembridge and Kostant.

Immanants, a class of matrix functions generalizing determinants and permanents, were first formulated by Schur during his fundamental paper \cite{S}. The term ``immanant'' was formally introduced by Littlewood and Richardson \cite{Li,LR} in their foundational study of representation theory of the symmetric groups.

Let  $\mathfrak{S}_n$ be the symmetric group  acting on the index set $\{1,\ldots,n\}$. For the complex character  $\chi^{\lambda}: \Sym_n\rightarrow \CC$  of the irreducible representation parameterized by partition $\lambda$ of $n$ and $n \times n$ complex matrix
 $A=(a_{ij})$, the $\lambda$-immanant of $A$ is defined by
\begin{equation}\label{eq:immanant}
    \mathrm{Imm}_{\chi^{\lambda}}(A) = \sum_{\sigma\in \mathfrak{S}_n}\chi^{\lambda}(\sigma)\prod_{i=1}^{n}a_{\sigma(i),i}.
\end{equation}

%A central problem for immanant involves establishing inequalities among these functions.
%A classical result is Fischer's inequality for positive semidefinite matrices \cite{F}: let
% $B$ is an $(n+p)\times (n+p)$ positive semidefinite Hermitian matrix partitioned as
%\[
%B = \begin{pmatrix}
%    B_{11} & B_{12} \\
%    B_{21} & B_{22}
%\end{pmatrix},
%\]
%where $B_{11}\in M_n(\mathbb{C})$ and $B_{22}\in M_p(\mathbb{C})$ are positive semidefinite. Then
%\begin{equation}\label{eq:fischer}
%    \det(B) \leq \det(B_{11})\det(B_{22}),
%\end{equation}
%the equality  if and only if  $B_{12}=0$ or  $B$ has a zero row/column.
For positive semidefinite Hermitian matrice $A$,
Schur  \cite{S} established a fundamental inequality:
\begin{equation}\label{eq:schur}
   \chi^\lambda(1) \det(A) \leq \mathrm{Imm}_{\chi^\lambda}(A),
\end{equation}
demonstrating that the determinant (corresponding to $\lambda=(1^n)$) is the minimal among all normalized immanants. Later, Stembridge \cite{St1} proved the same inequality for totally nonnegative matrices (real matrices with nonnegative minors).
He conjectured a stronger inequality.
%(No correspondingly strong statement holds for positive definite matrices).
\begin{conjecture}\cite{St2}
    If $A$ is a totally nonnegative $n\times n$ matrix and $\lambda\vdash n$, then
    \begin{equation}
        \Imm_{\phi^{\lambda}}(A)\geq 0,
    \end{equation}
    where  $\Imm_{\phi^{\lambda}}$ is the monomial immanant (i.e. the image of $\phi^{\lambda}$  is the monomial symmetric function $m_{\lambda}$,  under the characteristic map from the algebra of class functions of $\Sym_n$ to the algebra of symmetric functions).
\end{conjecture}

Let \( h_n \) denote the complete homogeneous symmetric function of degree \( n \) in countably many variables \( X = (x_1, x_2, \ldots) \), \( \mu \) and \( \nu \) be partitions with \( n \) parts. The Jacobi--Trudi matrices are of the form
\[
H_{\mu/\nu}(X) = [h_{\mu_i - \nu_j}(X)]_{i,j=1}^n.
\]
Since the minors of a Jacobi-Trudi matrix are skew Schur functions, they are
nonnegative in the sense of being nonnegative linear combinations of Schur functions.
Haiman \cite{Haiman} showed that \( \operatorname{Imm}_{\lambda}(H_{\mu/\nu}) -  \chi^\lambda(1) \det(H_{\mu/\nu}) \) is a nonnegative linear combination of Schur functions.

A longstanding open problem is  to determine the maximal  immanant.
In 1966, Lieb \cite{Lieb} conjectered that
\begin{conjecture}\cite{Lieb}
    For any positive semidefinite Hermitian $A $ and partition $\lambda \vdash n$, the permanent dominates all $\lambda$-immanants:
    \begin{equation}\label{eq:lieb}
        \chi^\lambda(1)\mathrm{per}(A) \geq \mathrm{Imm}_{\chi^\lambda}(A).
    \end{equation}
\end{conjecture}
This conjecture has been proved for $n \leq 13$ \cite{Pate1, Pate2}.  A comprehensive discussion of Lieb's permanent dominace conjecture   and related conjectures can be found in \cite{Wan,Z}.
%Given any  either positive definite or non-singular totally positive matrix $A\in \Mat_{n\times n}(\CC)$, from  Schur's and Stembridge's inequalities, we know that the $\lambda$-immanants of its any generalized principal  submatrix $A_I$ are  non-negative, where $I$ is any multiset of $[n]$.

Based on the inequalities established by Schur and Stembridge, it is known that the immanant of a singular matrix (whether positive semidefinite or totally nonnegative) is nonnegative. Knowing whether they are nontrivial or not will provide further insight to Lieb's conjecture.
In this work, we first provide a sufficient condition for the immanant of such a singular matrix to be {\it strictly positive}, which then strengthens
Schur and Stembridge's results.

On the other hand, Kostant \cite{Ko} has focused on an interpretation of immanants based on 0-weight spaces for finite dimensional representations of $\mathrm{SL}(n,\CC)$  and gives a different proof of the inequalities of Schur and  Stembridge \cite{S,St1} as an application.
Kostant's formulation also provides a trace formula for the $\lambda$-immanants without repeating indices.
Inspired by the work of Kostant,
we recently generalized the trace formula to quantum coordinate algebra in \cite{JLZ}.

In the second part, with the help of the Schur-Weyl duality, we will give a weight space interpretation for immanants of any generalized principal  submatrix of complex matrix in the general situation, i.e. for all weight space situation.
The new trace formula leads to a new criterion on the nonvanishing of $\Imm_{\chi^{\lambda}}(A_I)$, and we show $\Imm_{\chi^{\lambda}}(A_I)=0$ or not
only depends on whether the dimension of the weight space corresponding to $I$ in the highest weight representation $U^{\lambda}$ of $\GL_n$ is nonzero. In other words,     $\Imm_{\chi^{\lambda}}(A_I)>0$ if and only if $\lambda$ dominates the weight corresponding to $I$.   In view of Stembridge's conjecture, this also provides a deeper understanding of immanants for totally nonnegative matrices.

Our last main result provides an immanant inequality for either positive semidefinite or totally nonnegative matrix, which generalized the Kostant inequality in $0$-weight space to the orbit $O_{\mu}$ of any weight space $\mu$ under the action of Weyl group.
In particular, Schur's and Stembridge's inequalities can be acquired as special cases.

The paper is organized as follows. In section 2, we obtain a trace formula for the immanants of any generalized principal  submatrix. In section 3, we give a criterion for the nonvanishing of immanants of certain  positive semidefinite Hermitian or totally nonnegative matrices. Finally,
we prove an immanant inequality for either positive semidefinite or  totally nonnegative matrix.

\section{Immanants of generalized submatrix}
\subsection{Schur-Weyl duality}
By the classical Schur-Weyl duality,  $(\mathbb{C}^{n})^{\ot m}$ is a multiplicity free $\GL_n\times  \Sym_m $-module and there exists a decomposition of the space of tensors
\beq\label{Schur-Weyl-decom}
(\mathbb{C}^{n})^{\ot m}\cong\bigoplus_{\lambda \vdash m\atop \text{depth}(\lambda) \leq n} U^{\lambda}\ot V^{\lambda},
\eeq
where $(\pi_{\lambda},U^{\lambda})$ is the simple left $\GL_n$-module of highest weight $\lambda$
and $(\rho_{\lambda},V^{\lambda})$ is the simple left $\Sym_m$-module indexed by $\lambda$.

Let $\{e_i\}$ be the basis of $\mathbb C^n$. Define an inner product $\langle \cdot | \cdot \rangle$ on  $(\mathbb{C}^{n})^{\ot m}$ by
\beq\label{inner-product}
\bal
\langle e_{i_1}\ot \cdots \ot e_{i_m}, e_{j_1}\ot \cdots \ot e_{j_m}  \rangle&=\prod_{1\leq k \leq m}\delta_{i_k,j_k},
\eal
\eeq
where $i_1, \ldots, i_m, j_1, \ldots, j_m\in \{1, 2, \ldots, n\}$.
We introduce $*$-operations for $\Sym_m $  and $\GL_n$ respectively. Let  $\si^{*}=\si^{-1} \in \Sym_m$. And denote by $A^*$  the conjugate transpose  of matrix $A$ in $\GL_n$.  They are involutive anti-automorphisms.
\begin{lemma}\label{adjoint-op}
  The $*$-operations for $\Sym_m$  and $\GL_n$  on the inner product space $(\mathbb{C}^{n})^{\ot m}$ afford a contravariant  inner product, i.e.
\beq
\bal
  &\langle x\cdot  e_{i_1}\ot \cdots \ot e_{i_m}, e_{j_1}\ot \cdots \ot e_{j_m}  \rangle\\
  &=\langle  e_{i_1}\ot \cdots \ot e_{i_m}, x^* \cdot e_{j_1}\ot \cdots \ot e_{j_m}  \rangle,
\eal
\eeq
where $x$ is any element in $\Sym_m$ or $\GL_n$.

\end{lemma}

Let $\Lambda=(\lambda_{ij})_{1 \leq i \leq n\atop 1\leq j \leq i}$  be the  Gelfand-Tsetlin patterns
 satisfying the relations \cite{GT}:
\beq
\lambda_{ij}\geq \lambda_{i-1,j}\geq \lambda_{i,j+1}, \ \ 1\leq j \leq i \leq n.
\eeq

Denote by $\{\xi_{\Lambda}\}$ the Gelfand-Tsetlin  basis of the $\mathfrak{gl}(n)$-module $U^{\lambda}$.
Let $E_{ij}$, $1\leq i,j\leq n$, be the basis elements of $\mathfrak{gl}(n)$ and $E=(E_{ij})_{n\times n}$. Recall the central element $C_n(u)$ of $U(\mathfrak{gl}_n)$ introduced by Capelli \cite{C}:
\begin{equation}
C_n(u)=\sum_{\si\in \Sym_n}sgn(\si)(u+E)_{\sigma(1),1}\cdots (u+E-n+1)_{\si(n),n}.
\end{equation}

These Capelli elements act on the Gelfand-Tsetlin  basis of $U^{\lambda}$ as scalar multiplications, see \cite{Mo}.
\begin{lemma}\label{central-char}
Let $\{\xi_{\Lambda}\}$ be the Gelfand-Tsetlin  basis of $U^{\lambda}$. Then
\beq
 C_k(u)\xi_{\Lambda}=(u+\lambda_{k1})(u+\lambda_{k2}-1)\cdots (u+\lambda_{kk}-k+1)\xi_{\Lambda},  \ \ \text{for} \ 1 \leq k \leq n.
\eeq
\end{lemma}

Let $\{v_{\Tc} | sh(\Tc)=\lambda\}$ be Young's orthonormal %the  Gelfand-Tsetlin
basis of $V^{\lambda}$, where $\Tc$ runs through standard Young tableaux of shape $\lambda$.
We can decompose \eqref{Schur-Weyl-decom} into basis vectors:
\beq\label{Schur-Weyl-decom2}
(\mathbb{C}^{n})^{\ot m}\cong\sum_{\lambda \vdash m,\atop \text{depth}(\lambda)\leq n}\sum_{(\Lambda,\Tc)} \mathbb C \xi_{\Lambda}\ot v_{\Tc}.
\eeq
 By Lemma \ref{adjoint-op}, these basis vectors in decomposition \eqref{Schur-Weyl-decom2} are normalized orthogonal, i.e.
$$\langle \xi_{\Lambda}\ot v_{\Tc}, \xi_{\Lambda'}\ot v_{\Tc'}\rangle =\delta_{\Lambda\Lambda'}\delta_{\Tc\Tc'}$$
for any Gelfand-Tsetlin patterns $\Lambda,\Lambda'$ and standard  Young tableaux $\Tc,\Tc'$.

An $n$-tuple $(a_1,\ldots,a_n)$ in $\mathbb{Z}_{\geq 0}^n$ such that %that is a solution of the equation $x_1+\cdots +x_k=m$
$a_1+\cdots +a_n=m$ is called  a weak composition of $m$ into $n$ parts or a weak $n$-composition of $m$.
It's clear that
the weights of $\GL_n$ in $(\mathbb C^n)^{\otimes m}$ are exactly weak $n$-compositions of $m$.
For any weak  $n$-composition $\mu$ of $m$, define the projection operator
$$\mathcal P_{\mu}: (\mathbb{C}^n)^{\ot m}\rightarrow ((\mathbb{C}^n)^{\ot m})_{\mu},$$
then $\mathcal P_{\mu}: U^{\lambda}\rightarrow (U^{\lambda})_{\mu}$ where $((\mathbb{C}^n)^{\ot m})_{\mu}$ and $(U^{\lambda})_{\mu}$ are the subspaces with weight $\mu$ as a $\GL_n$-module.

It is convenient to write $I=(1^{m_1},\ldots,n^{m_n})$ to specify the multiplicity $m_i$ of $i$ in the multiset $I=(i_1,\ldots,i_m)$. Here $m_i=m_i(I)= \mathrm{Card}\{j\in I |j=i\}$.
Let  $$m({I})=m_1!m_2 ! \cdots m_n !.$$  Denote by $$\mathfrak S_{I}=\mathfrak S_{m_1}\times\mathfrak S_{m_2}\times \cdots \times \mathfrak S_{m_n}$$  the Young subgroup  of $\mathfrak{S}_m$, and $\mathfrak S_{m_j}=1$ if $m_j=0$.

Let $\lambda\vdash n$, the irreducible representation $V^{\lambda}=\{v_{\Lambda}| sh(\Tc)=\lambda\}$ of the symmetric group $\mathfrak S_n$ is realized by the following action:
\begin{equation}
s_iv_{\Lambda}=\frac1{c_{i+1}-c_i}v_{\Lambda}+\frac{\sqrt{(c_{i+1}-c_i)^2-1}}{c_{i+1}-c_i}v_{s_i\Lambda}
\end{equation}
where $c_i$ is the content of node $i$ in $\Lambda$.
The primitive idempotents can be constructed by semistandard tableaux. Let $\Tc$ be a semistandard tableau of shape $\lambda$, the idempotent
\begin{equation}
\mathcal E^{\lambda}_{\Tc}=\frac1{\prod_{(i,j)\in \Tc} h(i, j)}\sum_{\sigma\in\mathfrak S_n}\langle \sigma^{-1} v_{\Tc}, v_{\Tc}\rangle \sigma,
\end{equation}
where  $h(i,j)$ are the hook length of the $(ij)$-box in the Young tableau $\Tc$ of shape $\lambda$.

We extend $\pi_{\lambda}$ to all $n \times n$ matrices $\Mat_n(\CC)=\mathfrak{gl}(n)$, then $(\mathbb{C}^n)^{\ot m}$ has a $\Mat_n(\CC)$-module structure.
The following is a generalization of Kostant's Theorem \cite[Thm.3]{Ko} and gives a representation theoretic interpretation of the immanant.

\begin{theorem}\label{q-Kostant-thm}
  Let $\lambda \vdash m$, $\mu$ be  a weak  $n$-composition of $m$, and let the multiset $I=(1^{\mu_1},2^{\mu_2},\ldots,n^{\mu_n})$ of $[n]$. Then
\beq\label{general-Kostant-iden}
\frac{\Imm_{\chi^{\lambda}}(A_I)}{m(I)}=tr(\mathcal P_{\mu}\ot 1)\circ\pi_{\lambda}(A)\circ \mathcal P_{\mu}|_{U^{\lambda}},
\eeq
for any matrix $A\in \Mat_n(\CC)$.
\end{theorem}
\begin{remark}
  For $\mu=(1^{n})$, we obtain  Kostant's theorem in \cite[Thm. 3]{Ko}.
\end{remark}

\subsection{Proof of Theorem \ref{q-Kostant-thm}}
First we recall some notations. For any partition $\lambda$, the set of all Young tableaux of shape $\lambda$ is denoted
as
$YT(\lambda)$.
The set of semistandard Young tableaux of shape  $\lambda$ with entries at most $n$ is denoted by $SSYT(\lambda, n).$
The set of standard Young tableaux of shape
$\lambda$ is denoted as  $SYT(\lambda).$

Let $\lambda$ be a partition of $m$ and $\mu$ be a weak composition of $m$.
It is well-known that the semi-standard Young tableau $\Tc_{\lambda\mu}$ of shape $\lambda$ and weight $\mu$, where  $\mu_i$ equals the number of $i$'s in $\Tc_{\lambda\mu}$, is in one-to-one correspondence to a Gelfand-Tsetlin pattern $\Lambda=(\lambda_{ij})_{1 \leq i \leq m\atop 1\leq j \leq i}$ such that
\beq\label{pattern-tableau-corr}
\bal
\lambda_{ij}& = \text{the number of entries }\le i\text{ in the }j\text{th row of }\Tc_{\lambda\mu},\\
\mu_i&=\sum_{k=1}^{i}\lambda_{ik}-\sum_{s=1}^{i-1}\lambda_{i-1,s}.
\eal
\eeq
We will briefly write $\Tc_{\mu}=\Tc_{\lambda\mu}$ if the shape of $\Tc$ is clear from the context.
By the correspondence \eqref{pattern-tableau-corr}, the dimension of $(U^{\lambda})_{\mu}$  is the number of the semi-standard   tableaux $\Tc_{\lambda\mu}$.
 We also let $\{\xi_{\Tc_{\mu}}\}$  be the Gelfand-Tsetlin  basis vectors in $U^{\lambda}$.

 Given any weight $\mu=(\mu_1,\ldots,\mu_n)$, which can be written as a multiset $I=(1^{\mu_1} ,\ldots   ,n^{\mu_n})=(i_1, \ldots, i_m)$
  or a weak  $n$-composition of $m$. We define  the following map
%For any $I=(1^{\mu_1} ,\ldots   ,n^{\mu_n})=(i_1, \ldots, i_m)$, where $\mu=(\mu_1,\ldots,\mu_n)$ is  the weak  $n$-composition of $m$.
%For any partition $\lambda\vdash m$, the set of all Young tableaux of shape $\lambda$ is denoted
%as
%$YT(\lambda)$.
%The set of semistandard Young tableaux of shape  $\lambda$ with entries at most $n$ is denoted by $SSYT(\lambda, n).$
%The set of standard Young tableaux of shape
%$\lambda$ is denoted as  $SYT(\lambda).$
\beq
\bal
\theta_{\mu}: SYT(\lambda) \rightarrow YT(\lambda)\\
\eal
\eeq
which maps any standard Young tableau $\Tc$ to a Young tableau $\theta_{\mu}(\Tc)$ that replaces each node $r$ in $\Tc$  by $i_r$ for $1 \leq r \leq m$.
Note that $SSYT(\lambda, n)\subset SYT(\lambda)$ is contained in $YT(\lambda)$.%$\theta_{\mu}$.

\begin{lemma}[{\cite[Lemma 3.7]{JLZ}}]\label{schur-weyl-corr}
Let $\lambda \vdash m$ and $\mu$ be a weak $n$-composition of $m$ written as the multiset  $I=(1^{\mu_1},2^{\mu_2},\ldots,n^{\mu_n})=(i_1 \leq \ldots \leq i_m)$ of $[n]$, then as a $ U(\mathfrak{gl}_n) \times \Sym_m $-module
\begin{align*}
\sqrt{\frac{h_{\lambda}}{m(I)}}\mathcal{E}_{\Tc}^{\lambda}\cdot e_{i_1}\ot \cdots \ot e_{i_m}=\left\{\begin{array}{cc}
c\cdot \xi_{\theta_{\mu}(\Tc)}\ot v_{\Tc}&\text{ if } \theta_{\mu}(\Tc)  \in SSYT(\lambda, n),\\
0&\text{ if }\theta_{\mu}(\Tc)\notin  SSYT(\lambda, n),
\end{array}\right.
\end{align*}
where $h_{\lambda}=\prod_{(ij)\in \Tc}h(i,j)$ and $c$ is a nonzero constant. Explicitly,
$c=1$, if $\Tc$ is the unique pre-image of $\theta_{\mu}(\Tc)$. And
$\sum_{i=1}^s||c_i||^2=1$ for $\theta_{\mu}(\Tc_1)=\cdots=\theta_{\mu}(\Tc_s)$ ($s>1$).
In particular,
assume $m \leq n$, then
\beq
\sqrt{h_{\lambda}} \mathcal{E}_{\Tc}^{\lambda}\cdot e_1\ot \cdots \ot e_m=\xi_{\Tc}\ot v_{\Tc} .
\eeq
\end{lemma}

\begin{proof}[Proof of theorem \ref{q-Kostant-thm}]
By Lemma \ref{schur-weyl-corr},  the right side of \eqref{general-Kostant-iden} equals to
\begin{align*}
&tr(\mathcal P_{\mu}\ot 1)\circ\pi(A)\circ \mathcal P_{\mu}|_{U^{\lambda}}\\
&=\langle  \sqrt{\frac{h_{\lambda}}{m(I)}}(\mathcal P_{\mu}\ot 1) \circ \pi(A) \circ\sum_{\Tc}\mathcal{E}_{\Tc}\cdot e_{i_1}\ot \cdots \ot e_{i_m}, \sqrt{\frac{h_{\lambda}}{m(I)}}\sum_{\Tc}\mathcal{E}_{\Tc}\cdot e_{i_1}\ot \cdots \ot e_{i_m}\rangle\\
&=\frac{h_{\lambda}}{m(I)}\langle  (\mathcal P_{\mu}\ot 1) \circ \pi(A) \circ\sum_{\Tc}\mathcal{E}_{\Tc}\cdot e_{i_1}\ot \cdots \ot e_{i_m}, e_{i_1}\ot \cdots \ot e_{i_m}\rangle.
\end{align*}

By the relation of characters and primitive idempotents of the symmetric group, one has that
\begin{align*}
&\frac{h_{\lambda}}{m(I)}  (\mathcal P_{\mu}\ot 1) \circ \pi(A) \circ\sum_{\Tc}\mathcal{E}_{\Tc}\cdot e_{i_1}\ot \cdots \ot e_{i_m}\\
&=\frac{1}{m(I)} (\mathcal P_{\mu}\ot 1)\circ\pi(A) \circ\sum_{\tau\in \Sym_I,\atop \mu \in \mathcal{M}(\Sym_m/\Sym_I)} \chi^{\lambda}(\mu^{-1}\tau) P_{\mu^{-1}} P_{\tau} \cdot e_{i_1}\ot \cdots \ot e_{i_m}\\
&=\frac{1}{m(I)} (\mathcal P_{\mu}\ot 1)\circ\pi(A) \sum_{\tau\in \Sym_I,\atop \mu \in \mathcal{M}(\Sym_m/\Sym_I)} \chi^{\lambda}(\mu^{-1}\tau)e_{i_{\mu_1}}\ot \cdots \ot e_{i_{\mu_m}}
\end{align*}
where $\mathcal{M}(\Sym_m/\Sym_I)$ refers to the coset representatives in the quotient group $\Sym_m/\Sym_I$. The above expression then equals to
\ben
\bal
&=\frac{1}{m(I)}  \sum_{(j_1,\ldots,j_m)}\sum_{\tau\in \Sym_I,\atop \mu \in \mathcal{M}(\Sym_m/\Sym_I)} \chi^{\lambda}(\mu^{-1}\tau)
(\mathcal P_{\mu}\ot 1) a_{j_1,i_{\mu_1}}\cdots a_{j_m,i_{\mu_m}}\ot  e_{j_1}\ot \cdots \ot e_{j_m}\\
&=\frac{1}{m(I)}  \sum_{\si \in \Sym_m}\sum_{\tau\in \Sym_I,\atop \mu \in \mathcal{M}(\Sym_m/\Sym_I)} \chi^{\lambda}(\mu^{-1}\tau) a_{i_{\si_1},i_{\mu_1}}\cdots a_{i_{\si_m},i_{\mu_m}}\ot
 e_{i_{\si_1}}\ot \cdots \ot e_{i_{\si_m}},
\eal
\een
where the first sum is over all sequence $(j_1\ldots,j_m)$ of $[n]$.

As the action of $\Mat_n(\CC)$ commutes with that of $\Sym_m$ on the tensor space $(\mathbb{C}^{n})^{\ot m}$, so
\beq
\bal
&tr(\mathcal P_{\mu}\ot 1)\circ\pi(A)\circ \mathcal P_{\mu}|_{U^{\lambda}}\\
&=\frac{h_{\lambda}}{m(I)}\langle  (\mathcal P_{\mu}\ot 1) \circ \pi(A) \circ\sum_{\Tc}\mathcal{E}_{\Tc}\cdot e_{i_1}\ot \cdots \ot e_{i_m}, e_{i_1}\ot \cdots \ot e_{i_m}\rangle\\
&=\frac{1}{m(I)} \sum_{\tau\in \Sym_I,\atop \mu \in \mathcal{M}(\Sym_m/\Sym_I)} \chi^{\lambda}(\mu^{-1}\tau)
 a_{i_1,i_{\mu_1}}\cdots a_{i_{m},i_{\mu_m}}\\
&=\frac{\Imm_{\chi^{\lambda}}(A_I)}{m(I)}.
\eal
\eeq
\end{proof}

\section{Immanant inequalities and nonzero criterion}
Let $\mu$ be a weak $n$-composition of $m$, then $\mu$ is a weight of $\GL_n$.  We now present a criterion for non-vanishing of immanants of generalized principal submatrix of positive definite Hermitian or nonsingular totally nonnegative matrices based on the dimensions of the weight spaces.

Schur \cite{S} and Stembridge \cite{St1} have shown that the immanant of a singular matrix (whether positive semidefinite or totally nonnegative) is nonnegative.
Our criterion further  strengthened their important results.

\begin{theorem}\label{criterion}
Let $A$ be any an $n\times n$  either positive definite or nonsingular  totally nonnegative  matrix, $\lambda \vdash m$ and $\mu$ is a weak  $n$-composition of $m$. Given  the multiset $I=(1^{\mu_1},2^{\mu_2},\ldots,n^{\mu_n})$ of $[n]$, then
\beq
\Imm_{\chi^{\lambda}}(A_I)>0
\eeq
if and only if  $\dim(U^{\lambda})_{\mu}\neq 0$.
\end{theorem}
\begin{proof}
It follows from Theorem \ref{q-Kostant-thm} that
the immanant of $A_I$ is zero if $\dim(U^{\lambda})_{\mu}= 0$.

Conversely, for $\dim(U^{\lambda})_{\mu}\neq 0$, we show that $\Imm_{\chi^{\lambda}}(A_I)>0$.
We simply write $e_{I_{\mu}}=e_{i_1}\ot \cdots \ot e_{i_m}$.
By Theorem \ref{q-Kostant-thm}, we have that

\beq\label{term1}
\bal
\Imm_{\chi^{\lambda}}(A_I)=\sum_{\Tc}h_{\lambda} \langle  \pi(A) \circ\mathcal{E}^{\lambda}_{\Tc}\cdot e_{I_{\mu}}, \mathcal{E}^{\lambda}_{\Tc}\cdot e_{I_{\mu}}\rangle,
\eal
\eeq
where  the  sum is taken over all standard Young tableaux $\Tc$ with $\theta_{\mu}(\Tc)  \in SSYT(\lambda, n)$.

If $A$ is positive definite, by the Cholesky decomposition, we may write $A=\overline{U}^t DU$, where $U$ is upper triangular and unipotent, $D=\diag(d_1,\ldots,d_n)  \in \GL_n$ is a diagonal matrix with positive entries. Thus
\beq
\bal
  &\sum_{\Tc}h_{\lambda} \langle  \pi(A) \circ\mathcal{E}_{\Tc}\cdot e_{I_{\mu}}, \mathcal{E}_{\Tc}\cdot e_{I_{\mu}}\rangle\\
  &=\sum_{\Tc}h_{\lambda}\langle \pi(DU) \circ \mathcal{E}_{\Tc} \cdot e_{I_{\mu}},\pi(U) \circ \mathcal{E}_{\Tc} \cdot  e_{I_{\mu}}\rangle\\
  &\geq\sum_{\Tc}d_{i_1}\cdots d_{i_m}h_{\lambda} \langle \mathcal{E}_{\Tc}\cdot e_{I_{\mu}},  \mathcal{E}_{\Tc}\cdot e_{I_{\mu}}\rangle\\
  &=d_{i_1}\cdots d_{i_m}m(I)\dim (U^{\lambda})_{\mu}.
\eal
\eeq

If $A$ is  nonsingular totally nonnegative, then by the Whitney theorem in \cite{W,Lo,Ko},  we can write $A=U_{-}U_{+}D$, where $U_{+}$ is a product of elements of the form $\exp tE_{i,i+1}$ with $t>0$, $U_{-}$ is a product of elements of the form $\exp tE_{i+1,i}$ with $t>0$, and $D=\diag(d_1,\ldots,d_n) \in \GL_n$ is a diagonal matrix with positive entries.
Thus if $\dim(U^{\lambda})_{\mu}\neq 0$,
\beq
\bal
  &\sum_{\Tc}h_{\lambda} \langle  \pi(A) \circ\mathcal{E}_{\Tc}\cdot e_{I_{\mu}}, \mathcal{E}_{\Tc}\cdot e_{I_{\mu}}\rangle\\
  &=\sum_{\Tc}h_{\lambda}\langle \mathcal{E}_{\Tc}\circ\pi(U_{-}U_{+}D) \cdot e_{I_{\mu}}, e_{I_{\mu}}\rangle\\
    &=\sum_{\Tc}h_{\lambda}d_{i_1}\cdots d_{i_m}\langle \mathcal{E}_{\Tc}\circ\pi(U_{-}U_{+}) \cdot e_{I_{\mu}}, e_{I_{\mu}}\rangle\\
    &\geq\sum_{\Tc}d_{i_1}\cdots d_{i_m}h_{\lambda} \langle \mathcal{E}_{\Tc}\cdot e_{I_{\mu}},  e_{I_{\mu}}\rangle\\
    &=d_{i_1}\cdots d_{i_m} m(I) \dim (U^{\lambda})_{\mu}>0.
\eal
\eeq
\end{proof}

Based on the inequalities established by Schur and Stembridge, the immanant of a singular matrix (whether positive semidefite or
totally nonnegative) is nonnegative. Theorem \ref{criterion} provides a
sufficient condition for the immanant  of  such a singular matrix to be positive.
%\begin{theorem}
%Let $\lambda \vdash m$ and $\mu$ is a weak  $n$-composition of $m$. Given  the multiset $I=(1^{\mu_1},2^{\mu_2},\ldots,n^{\mu_n})$ of $[n]$, assume that $A$ be any an $n\times n$  either positive semidefinite or  totally nonnegative matrix such that $A$ is a generalized principal
%submatrix of some positive definite Hermitian or nonsingular totally nonnegative
%matrix,  then if $\dim(U^{\lambda})_{\mu}\neq 0$,
%\beq
%\Imm_{\chi^{\lambda}}(A)>0.
%\eeq
%\end{theorem}

Denote by $O_{\mu}$ the orbit of weight $\mu$ under the action of the Weyl group $\Sym_n$. Let $I=(1^{\mu_1},2^{\mu_2},\ldots,n^{\mu_n})$ be the multiset of $[n]$ corresponding to $\mu$. Finally, we have the following generalization of Kostant's Theorem 4 in \cite{Ko}.
\begin{theorem}
Let  $\lambda \vdash m$, $\mu$ is a weak  $n$-composition of $m$, then
\beq
\sum_{\nu \in O_{\mu}}tr(\mathcal P_{\nu}\ot 1)\circ\pi(A)\circ \mathcal P_{\nu}|_{U^{\lambda}}\geq |O_{\mu}|\dim (U^{\lambda})_{\mu}\det(A),
\eeq
whenever $A$ is either positive semidefinite or  totally nonnegative.
\end{theorem}
\begin{proof}
It suffices to show that
\beq
\bal
   \sum_{\nu \in  O_{\mu}} \sum_{\Tc} \frac{h_{\lambda}}{m(I)} \langle  \pi(A) \circ\mathcal{E}^{\lambda}_{\Tc}\cdot e_{I_{\nu}}, \mathcal{E}^{\lambda}_{\Tc}\cdot e_{I_{\nu}}\rangle\geq  |O_{\mu}|\dim (U^{\lambda})_{\mu}\det(A).
\eal
\eeq
From the proof in theorem \ref{criterion}, we know that
\beq
\bal
    \sum_{\nu \in  O_{\mu}}\sum_{\Tc}\frac{h_{\lambda}}{m(I)} \langle  \pi(A) \circ\mathcal{E}_{\Tc}\cdot e_{I_{\nu}}, \mathcal{E}_{\Tc}\cdot e_{I_{\nu}}\rangle
   \geq \sum_{\nu \in O_{\mu}}d_{i_1}\cdots d_{i_m}\dim (U^{\lambda})_{\nu},
\eal
\eeq
where  $D=\diag(d_1,\ldots,d_n)$ is a diagonal matrix. For any $\nu_1,\nu_2\in O_{\mu}$, there exists an element $w\in \Sym_n$ such that $w\cdot \nu_1=\nu_2$ and $\dim(U^{\lambda})_{\nu_1} =\dim(U^{\lambda})_{w\cdot \nu_1} =\dim(U^{\lambda})_{\nu_2}$. By the inequality of arithmetic and geometric means,
\beq
\bal
    \sum_{\nu \in  O_{\mu}}d_{i_1}\cdots d_{i_m}
    &\geq |O_{\mu}|  \prod_{\nu \in O_{\mu}}(d_{i_1}\cdots d_{i_m})^{\frac{1}{| O_{\mu}|}}.
\eal
\eeq
Note that  $\prod_{\nu \in O_{\mu}}(d_{i_1}\cdots d_{i_m})$ is a symmetric polynomial in the variables $d_1,\ldots,d_n$.
Hence $\prod_{\nu \in O_{\mu}}(d_{i_1}\cdots d_{i_m})^{\frac{1}{| O_{\mu}|}}= d_1\cdots d_n=\det(A)$. Moreover,
\beq
    \sum_{\nu \in  O_{\mu}}\sum_{\Tc}\frac{h_{\lambda}}{m(I)} \langle  \pi(A) \circ\mathcal{E}_{\Tc}\cdot e_{I_{\nu}}, \mathcal{E}_{\Tc}\cdot e_{I_{\nu}}\rangle
    \geq| O_{\mu}|\dim (U^{\lambda})_{\mu}\det(A).
\eeq
\end{proof}

\bigskip
\centerline{\bf Acknowledgments}
\medskip
The work is supported in part by the National Natural Science Foundation of China grant nos.
12171303 and 12001218, the Humboldt Foundation, the Simons Foundation grant no. 523868,
and the Fundamental Research Funds for the Central Universities grant nos. CCNU24JC001, CCNU25JC025 and CCNU25JCPT031.

\bibliographystyle{amsalpha}

\end{document}